\documentclass[11pt]{article}
\usepackage{longtable,geometry}
\usepackage{graphicx}
\usepackage{amsmath, amsthm, amsopn, amssymb, enumerate}
\newtheorem{thm}{Theorem}
\newtheorem{lemma}[thm]{Lemma}

\newtheorem{cor}[thm]{Corollary}

\newtheorem*{thm*}{Theorem}
\newtheorem*{lemma*}{Lemma}

\theoremstyle{definition}

\renewcommand{\Pr}{\mathbb P}

\newcommand{\Z}{\mathbb{Z}}

\newcommand{\dist}{\rho}

\newcommand{\eps}{\varepsilon}

\newcommand{\set}[1]{\left\{ #1 \right\}}
\newcommand{\IDLA}{{\sf IDLA}}

\newcommand{\until}{\mapsto}

\author{Hugo Duminil-Copin \ \ \ Cyrille Lucas \ \ \ Ariel Yadin \ \ \ Amir Yehudayoff}

\title{Containing Internal Diffusion Limited Aggregation}

\begin{document}
\date{}

\maketitle
\begin{abstract}
Internal Diffusion Limited Aggregation (IDLA) is a model that describes
the growth of a random aggregate of particles from the inside out.
Shellef proved that IDLA processes
on supercritical percolation clusters of integer-lattices 
fill Euclidean balls, with high probability. 
In this article, we complete the picture and
prove a limit-shape theorem for \IDLA\ on such percolation clusters,
by providing the corresponding upper bound. 

The technique to prove upper bounds is new and robust: it only requires the existence of a ``good'' lower bound. Specifically, this way of proving upper bounds on \IDLA\ clusters is more suitable for random environments than previous ways, since it does not harness harmonic measure estimates.
\end{abstract}

\section{Introduction}


The Internal Diffusion Limited Aggregation (\IDLA) model was introduced by Diaconis and Fulton in \cite{DF}, and gives a protocol for recursively building a random aggregate 
of particles. At each step, the first vertex visited outside the current aggregate by a random walk started at the origin is added to the aggregate. 
This simulates the growth of an aggregate of particles from the inside out.

In a number of settings, this model is known to have a deterministic limit-shape,
meaning that a random aggregate with a large number of particles has a typical shape.
On $\Z^d$, Lawler, Bramson and Griffeath \cite{lawler1992internal} were the first to identify this limit-shape, in the case of simple random walks, as an Euclidean ball. Their result was later sharpened by Lawler \cite{lawler1995subdiffusive}, and was recently drastically improved with the simultaneous works of Asselah and Gaudill\`ere \cite{asselah2010logarithmic, asselah2010sub} and Jerison, Levine and Sheffield \cite{jerison2010logarithmic,jerison2011internal,jerison2010internal}. On other graphs, current knowledge is less precise. On groups with polynomial growth, the existence of the limiting shape is unknown, although Blach\`ere gave bounds on the cluster \cite{blachere2004internal}. On finitely generated groups with exponential growth, with a suitable metric, a limit-shape result was proved by Blach\`ere and Brofferio \cite{blachere2007internal}. Huss also studied \IDLA\ for a large class of random walks on general non-amenable graphs in \cite{huss13internal}.

Another interesting question about IDLA is whether the limit-shape is robust
to small perturbations of the underlying graph;
For example, on the infinite cluster of supercritical percolation cluster of $\Z^d$.
Shellef proved a sharp inner bound for the \IDLA \ model on the infinite cluster \cite{shellef15idla}. Figure \ref{simuperco} presents the \IDLA \ aggregate built on the supercritical bond percolation cluster.

\begin{figure}[ht]
\centering \label{simuperco}
\includegraphics[scale=0.65]{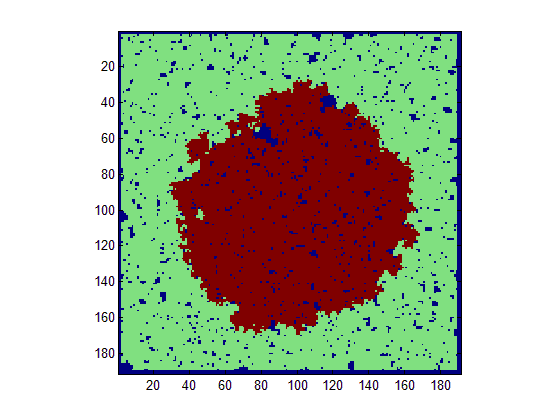}
\caption{\textit{\IDLA \ aggregate with 6,300 particles on the supercritical percolation cluster on $\Z^2$ (edges are deleted w.p. 0.4). Red points are in the aggregate, green points are in the cluster, and blue points are outside the cluster.}}
\end{figure}

To prove the existence of a limit-shape on the supercritical percolation
cluster on $\Z^d$,
it thus remains to show a sharp outer bound on \IDLA.
We provide such an outer bound, relying on Shellef's inner bound. 
\begin{thm}\label{perco}
Fix $d>0$ and $p>p_c(\Z^d)$ and let $\omega$ be the infinite cluster of 
percolation on $\Z^d$ with parameter $p$, 
conditioned on the origin $o$ belonging to $\omega$. 
Let $B_o(n)$ be the Euclidean ball of radius $n$ centered at the origin.
Let $b_o(n) = | \omega \cap B_o(n)|$, the size of the part of the cluster 
that is in $B_o(n)$.
Let $A_{b_o(n)}(o)$ be $\IDLA $ generated by $b_o(n)$ particles started at the origin.
Then, for every $\varepsilon > 0$  and a.s. every $\omega$, the \IDLA\ on $\omega$ a.s. satisfies:
$$B_o((1-\varepsilon) n)\subset A_{b_o(n)}(o)\subset B_o((1+\varepsilon) n) \qquad\forall \ n\text{ large enough}.$$
\end{thm}

More generally, Theorem~\ref{thm:main thm} below states that given a sharp inner bound on 
\IDLA \ for a graph (with some regularity assumptions), a sharp outer bound exists
as well. 
The inner and outer bounds together prove the almost sure convergence to a limit-shape. 
The regularity assumptions are quite mild,
and the theorem should be useful in great generality. 
Since the percolation cluster satisfies the necessary regularity assumptions,
Theorem~\ref{thm:main thm} implies Theorem~\ref{perco}.
We also mention other environments that
satisfy the regularity assumptions, like Cayley graphs with polynomial growth
(more details follow).

Let us mention that, despite the fact that the upper bound on $\Z^d$ is not intrinsically harder than the lower bound, they invoke different ingredients. The lower bound usually harnesses estimates on the Green function, while the upper bound requires the use of upper bounds on harmonic measures. The Green function is fairly well understood in several (random) environments such as supercritical percolation, random conductances with elliptic condition, and so forth. This is not the case for harmonic measures. 
The techniques developed in this article allow to bypass this difficulty. 
Our main result relates an upper bound on \IDLA \ directly to
a lower bound.  Roughly speaking, we show that if one knows that balls of a certain metric 
are contained in the aggregate, and not many particles are left over,
then one can deduce an upper bound and shape theorem.

\subsection{Definition of IDLA}

Let $G$ be a graph.
Let $S \subset G$ be a finite subset of $G$. 
In order to define \IDLA, first define adding one 
particle started at $x$ to an existing aggregate $S$.
For $x \in G$, denote by $A(S;x)$ \
the \IDLA \ aggregate obtained as follows:  Let $\xi = (\xi(0),\xi(1),\ldots)$
be a random walk on $G$ started at $\xi(0) = x$ and let $t_S$ be the first time this walk is not in $S$.  Define
$$A(S;x) := S \cup \set{\xi(t_S)}.$$

It is standard to consider a slightly more general process, where the growth of the aggregate is stopped at certain stopping times,
{\em e.g.}, upon exiting a set $T$.
Denote by $A(S;x \until T)$ 
the aggregate obtained by letting a particle
randomly walk from $x$, but pausing it if it exits $T$, which is defined as follows. 
Let $\xi$ be a random walk on $G$ started at $x$. 
Let $t_S$ be the first time this walk is not in $S$ as above,
and let $t_T$ be the first time $\xi$ exits $T$.
Define $$A(S;x \until T) := S \cup \set{ \xi(t_S \wedge (t_T-1)) }.$$
To keep track of the position of a paused particle, define 
$$P(S;x \until T) = \left\{
\begin{array}{ll}
\xi(t_T-1) & \text{if $t_T < t_S$,}\\
 \perp & \text{ otherwise,}
\end{array}
\right.$$
so $\perp$ means that the particle is already ``absorbed''
in the aggregate.


Given vertices $x_1,\ldots,x_k$ in $G$ and a set $T$, define $A(S;x_1,\ldots,x_k\until T)$ to be the \IDLA \ aggregate formed from an existing aggregate $S$ by $k$ particles, started at $x_1,\ldots,x_k$, and  paused upon exiting $T$.
That is, define inductively:
$S_0 = S$, 
$S_j = A(S_{j-1} ,x_j\until T)$ for $j \in \{1,\ldots,k\}$, 
and $A(S;x_1,\ldots,x_k \until T) = S_k$.
Again, to keep track of paused particles, 
define $P(S;x_1,\ldots,x_k \until T)$ to be the
sequence of particles paused in this process; 
Formally,
if $p_j = P(S_{j-1}; x_j  \until T)$ for $j \in \{1,\ldots,k\}$,
then $P(S;x_1,\ldots,x_k \until T)$ 
is the sequence $(p_j : p_j \neq \perp )$. When particles are not stopped, we define the aggregate similarly and we denote it by $A(S;x_1,\ldots,x_k)$.

One reason to keep track of these paused particles 
is the so called {\em Abelian property} of IDLA:
\begin{align}
A(S;x_1,\ldots,x_k)  \ \ &\textrm{has}  \textrm{ the same distribution as} \nonumber\\
  &A\big( A(S;x_1,\ldots,x_k \until T);P(S;x_1,\ldots,x_k \until T) \big) .\label{eqn:abelian property}
\end{align}
Equation \eqref{eqn:abelian property}
says that in order to sample $A(S,x_1,\ldots,x_k)$,
one can sample $A(S;x_1,\ldots,x_k \until T)$ while pausing particles
upon exiting $T$ and keeping track of them via $P(S;x_1,\ldots,x_k \until T)$,
and then restart the paused particles on the obtained aggregate
$A(S;x_1,\ldots,x_k \until T)$.
For more details, see \cite{DF, lawler1992internal}.

We are mostly interested in $n$ particles starting at just one point.
For an integer $n \geq 0$,
by $A_n(x)$ we denote the \IDLA \ aggregate with $n$ particles started at $x$,
that is, 
$$A_n(x) = A
(\emptyset;\underbrace{x,\ldots,x}_{\text{\tiny $n$ times}}) .$$

We also focus on pausing particles according to a metric $\rho$. Define the {\em ball} of radius $r$ around $x$ to be
$B_x(r) = \set{ y \ : \ \rho(x,y) < r}$,
and denote its size by $b_x(r) = |B_x(r)|$. As above, set
$$A_n(x\until r) = A(\emptyset;x,\ldots,x\until B_x(r))$$
and
$$P_n(x\until r) = P(\emptyset;x,\ldots,x\until B_x(r)).$$

\subsection{Assumptions}

We now make a few assumptions. The first two assumptions are independent of the \IDLA \ process. We will always compare the \IDLA\ cluster to a ball for a certain metric $\rho$ on $G$, on which we now make a few hypotheses. Assume that $\rho$ satisfies the following assumptions:
\medbreak
\noindent
\textbf{Continuity $(C)$:} the metric is dominated by the graph distance $d_G$ on $G$: 
there exists $c>0$ such that
\begin{equation}\nonumber
\rho(x,y) \le c\cdot d_G(x,y) \quad \forall \ x,y\in G.
\end{equation}

\medbreak
\noindent
\textbf{Regular volume growth $(VG)$:} there exist $c,d>0$ such that for every $n>0$,
\begin{equation}\nonumber
\label{eqn:volume growth}
 \frac1cr^d \le |B_x(r)| \le cr^d\quad\quad\forall \ 
x\in B_o(n)\text{ and }n^{1/d^3}\le r\le n.
\end{equation}
The conditions on $r$ and $n$ allow to consider inhomogeneous environments. Keep in mind that the percolation cluster contains arbitrary finite graphs infinitely often.


We will also make the following assumption on \IDLA:
\medbreak
\noindent
\textbf{Weaker lower bound $(wLB)$:} there exists $\alpha >0$ such that
for every $n>0$,
\begin{align}\nonumber
\label{eqn:reasonable lower bound}
\Pr[ B_x(r) \subset A_{b_x(r / \alpha )}(x \until r)]
\ge \alpha \quad \quad \forall \ x \in B_o(n + r)\text{ and }n^{1/d^3}\le r \le n.
\end{align}
In words, with noticeable probability,
when releasing order $b_x(r)$ particles at $x$ the aggregate contains
$B_x(r)$. This assumption is easy to verify for many \IDLA \ processes. In order to prove the existence of a limiting shape, we will assume a stronger lower bound for the aggregate grown around the origin, as we will see in the next section.



\subsection{Main results}

The first theorem relates a lower bound to an upper bound
in well-behaved environments. 



\begin{thm}
\label{thm:main thm}
Let $G$ be a graph and $\rho$ be a metric on $G$ satisfying conditions $(C)$, $(VG)$ and $(wLB)$. Then, there exists a constant $c_1  > 0$ such that
for any $\eps >0$,
\begin{align*}
\Pr \big[ A_{b_o(n)}(o) \not\subset B_o((1+ \eps) n) \ \ i.o. \big] & \leq
\Pr \big[ |A_{b_o(n)}(o\until n)| < b_o(n) (1 - c_1 \eps^d)  \ \ i.o.  \big] .
\end{align*}
\end{thm}

The previous theorem is especially useful when a lower bound is known.  Indeed, although we need a statement stronger than a simple lower bound, the usual proofs of lower bounds always yield the fact that the \IDLA\ process (almost) fills all large enough balls when stopped on their boundary:

\medbreak\noindent\textbf{Lower bound $(LB)$:} $|A_{b_o(n)}(o\until n)|/b_o(n)$ converges almost surely to 1. 
\medbreak
The difference between (LB) and the ``usual'' notion of
a lower bound is that the ``usual'' lower bound says
that the aggregate {\em eventually} contains a large ball, 
whereas (LB) says that the aggregate contain a large ball,
even when stopped upon exiting a (slightly bigger) ball.


\begin{cor}
Let $G$ be a graph and $\rho$ a metric on $G$ satisfying conditions $(C)$, $(VG)$, $(wLB)$, and $(LB)$, then for every $\varepsilon>0$, a.s.
$$A_{b_o(n)}(o)\subset B_o((1+\varepsilon) n)\quad \quad \forall  \
n \ \text{large enough}.$$
\end{cor}

Even though the main application of the theorem above and its corollary will be in the case of supercritical percolation, they apply in great generality and we believe that they will be useful in future works on \IDLA.



\section{An upper bound on \IDLA \ on percolation clusters}

Let us prove that Theorem~\ref{thm:main thm} and its corollary imply Theorem~\ref{perco}. In this case, $\rho$ will be the Euclidean distance. 
By Theorem~\ref{thm:main thm}, we only need to check that $(C)$, $(VG)$, $(wLB)$ and $(LB)$ are satisfied: 
\begin{description}
\item[(C)] Since $\rho\le d_\omega$ deterministically, property $(C)$ is satisfied for every $\omega$ with the constant $1$. 

\item[(VG)] Barlow proved in \cite{barlow} that $(VG)$ is satisfied for the distance $d_\omega$ for almost every environment. A classical result of \cite{antal1996chemical} easily implies that for almost every environment $\omega$, there exists $c_1=c_1(\omega)>0$ such that for every $n>0$,
\begin{equation}\label{comparison distances}
\rho(x,y)\le d_\omega(x,y) \le c_1 \rho(x,y)\quad\forall \ x,y\in B_o(n):c_1\log n\le\rho(x,y)\le n
\end{equation}
so that $(VG)$ is also satisfied for $\rho$ with a possibly different constant (the result also follows from \cite{barlow}). 

\item[(wLB)] In \cite{shellef15idla}, Shellef proved that for any $\varepsilon>0$, there exists $\eta>0$ such that the following holds for almost every environment $\omega$: there exists $c_2=c_2(\omega)>0$ such that
\begin{equation*}
\Pr\left[B_o((1-\varepsilon) n)\subset A_{b_o(n)}(o\until n)\right]\ge 1-\frac {c_2}{n^{d+2}}
\end{equation*} and $\mathbb P_p[c_2\ge \lambda]\le e^{-\lambda^\eta}$ for all $\lambda>0$. All together, this implies that for almost every environment $\omega$, there exists $c_3=c_3(\omega)>0$ such that for every $n>0$,
\begin{equation}\label{crucial}
\Pr\left[B_x((1-\varepsilon) r)\subset A_{b_x(r)}(x\until r)\quad\forall x\in B_o(n)\text{ and }n^{1/d^3}\le r\le n\right]\ge 1-\frac {c_3}{n}.
\end{equation} 
The condition $(wLB)$ follows readily. The result in Shellef deals with the event $B_o((1-\varepsilon) n)\subset A_{b_o(n)}(o)$ (particles are not stopped at distance $n$), but the proof actually implies this stronger result.

\item[(LB)] Finally, note that the comparison between distances \eqref{comparison distances} and $(VG)$ imply that $$b_o(n)-b_o((1-\varepsilon) n)~\le~ c_4\varepsilon^d b_o(n)$$
for some constant $c_4=c_4(\omega)>0$ depending on the environment, so that \eqref{crucial} implies $(LB)$ for almost every environment.
\end{description}



\section{Proof of Theorem~\ref{thm:main thm}}

From now on, we fix a graph $G$ satisfying $(C)$, $(VG)$ and $(wLB)$. Constants in the proof always depend only on the constants involved in $(C)$, $(VG)$ and $(wLB)$, {\em i.e.} $c,d$ and $\alpha$.

The following lemma shows that a lower bound on the aggregate
implies a lower bound on hitting probabilities. It is a general statement not invoking any of the conditions $(C)$, $(VG)$ or $(wLB)$. In the following, we make a slight abuse of notations: $\xi$ will denote a random walk as well as its trace.


\begin{lemma}
\label{lem:hit prob}
Let $Q \subset B \subset G$ and $x\in B$.
Let $\xi$ be a random walk started at $x$ and stopped on exiting $B$.
For any $t>0$,
$$\Pr [ \xi \cap Q \neq \emptyset ]  \geq
\Pr [ B \subset A_t(x\until B) ] \cdot  |Q| / t.$$
\end{lemma}
\noindent The above lemma is most useful when $t$ is chosen so that
$\Pr [ B \subset  A_t(x\until B) ]$ is of order $1$. 

\begin{proof}
Let $\xi_1,\ldots,\xi_t$ be the $t$ independent random walks started at $x$
and stopped on exiting $B$
that generate the aggregate $ A_t(x\until B)$.
Let $J \in \{1,\ldots,t\}$ be a uniformly chosen index independent of the random walks.
Consider the set 
$\Gamma$ of $j \in \{1,\ldots,t\}$ so that
$\xi_j$ hits $Q$ before exiting $B$.
Since $Q \subset B$, 
the inclusion $B \subset  A_t(x\until B)$ implies $|\Gamma| \geq |Q|$.
Since $J$ is independent of $\xi_1,\ldots,\xi_t$,
\begin{align*}
\Pr [ \xi_J \cap Q \neq \emptyset \ | \ B \subset  A_t(x\until B) ] \geq
\Pr [ J \in \Gamma \ | \ B \subset  A_t(x\until B) ] 
\geq |Q|/t .
\end{align*}
The lemma follows since the distribution of $\xi_J$
is that of a random walk started at $x$ and stopped when exiting $B$.
\end{proof}

By assumption on $G$,
we can hence get the following hitting probability estimate, which states that a random walk hits a set, whose complement has size at most $\epsilon r^d$, with a probability that is bounded away from zero.

\begin{lemma}\label{lem:exit prob}
There exist $\epsilon,\eta >0$ such that 
for large enough $n$ and $ n^{1/(d(d+1))}<r<n$, the following holds.
Let $x \in B_o(n)$ and let $S \subset B_o(n+r)$ be so that
$|S \setminus B_o(n) | \leq \epsilon r^d$.
Let $\xi$ be a random walk started at $x$ and
stopped upon exiting $B_o(n+r)$.
Then,
$$\Pr \left[\xi\cap \big(B_o(n+r) \setminus 
(S\cup B_o(n))\big) \neq \emptyset \right] \ge \eta .$$
\end{lemma}

\begin{proof}
For every path $\gamma$ from
inside $B_o(n)$ to outside $B_o(n+r)$,
let $y(\gamma)$ be the first vertex on $\gamma$
so that $\dist(y(\gamma),B_o(n)) \geq r/2$.
Denote by $Y$ the set of all $y(\gamma)$ for such paths $\gamma$.
Every path from $x$ to outside $B_o(n+r)$ must hit $Y$.
By Markov's property,
it thus suffices to prove the theorem for starting points $y \in Y$.
Fix $y \in Y$.

Let $B = B_y(r/3)$ and $Q= B \setminus S$. 
By $(VG)$ and by assumption on $S$,
$$|Q| \ge \frac1c(r/3)^d  -\epsilon r^d
\ge \frac 1c(r/4)^d ,$$
with $\epsilon = 4^{-d}/c$.
By $(wLB)$
with $t = b_y(r/(3\alpha))$,
$$\Pr[ B \subset A_t(y \until r/3)] \ge \alpha .$$
Let $\xi$ be a random walk started at $y$ and stopped on exiting $B$.
Lemma~\ref{lem:hit prob} and $(VG)$ imply that
$$\Pr [ \xi \cap Q \neq \emptyset ]  \geq \alpha \frac{(r/4)^d}{c b_y(r/(3\alpha))} 
\geq \alpha (\alpha/4)^{d}/ c^2 = :\eta .$$
Note that 
$r/2\le \rho(y,B_o(n)) \leq r/2 + c$ thanks to the definition of $Y$ and $(C)$.
Therefore, $\rho(y,G \setminus B_o(n+r)) \geq r - r/2 - c > r/3$
for $n$ large, and so $B \subset B_o(n+r) \setminus B_o(n)$. We deduce
$$\Pr \left[\xi\cap \big(B_o(n+r) \setminus 
(S\cup B_o(n))\big) \neq \emptyset \right] \ge 
 \Pr [ \xi \cap Q \neq \emptyset ]\ge \eta.$$
\end{proof}

After analyzing the behavior of a single particle,
we can analyze the behavior of the whole aggregate.
The following lemma says that, with high probability, 
a constant fraction of the aggregate is absorbed
in a wide enough (yet still very fine) annulus.

\begin{lemma}\label{lem:fundamental}
There exist $\delta>0$ and $p < 1$ such that
for all $n$ large enough, for all $n^{1/(d+1)}<k<n$
and $x_1,\ldots,x_k \in B_o(n)$, and for all $S \subset B_o(n)$,
$$\Pr\big[ |A(S;x_1,\ldots,x_k  \until B_o(n+k^{1/d}) )\setminus S| \le \delta k\ \big] \le p^k.$$
\end{lemma}

\begin{proof}
Let $r=k^{1/d}$. Fix $\epsilon,\eta$ as in Lemma~\ref{lem:exit prob}.
Let $\xi_1,\ldots,\xi_k$ be the random walks 
started at $x_1,\ldots,x_k$ that generates the aggregate.
Let $k' = \lfloor \epsilon k \rfloor \leq \epsilon r^d$.
For $j \in \{1,\ldots,k'\}$, denote
$$A_j= A(S;x_1,\ldots,x_j\until B_o(n+r) ).$$
Since $|A_j\setminus B_o(n)|\le j\le \varepsilon r^d$, Lemma~\ref{lem:exit prob} implies that for all $j\in\{1,\ldots,k'\}$,
$$\Pr\left[\xi_{j+1} \cap \big(B_o(n+r)\setminus A_j\big) \neq \emptyset\ |\ A_j\right] \ge \eta.$$
Therefore, $|A(S;x_1,\ldots,x_k \until B_o(n+r)) \setminus S|$ dominates a 
$(k',\eta)$-binomial random variable. 
Thus, there exist $\delta > 0$ and $p <1$
depending only on $\epsilon,\eta$ such that
$$\Pr\big[|A(S;x_1,\ldots,x_k \until B_o(n+r)) \setminus S | \le \delta k\big] \le p^{k}.$$
\end{proof}

We now turn to the proof of Theorem \ref{thm:main thm}.
The proof consists of inductively constructing a sequence of aggregates $A_j$
by pausing the particles at different distances $n_j$ from the origin.
If $k_j$ is the number of paused particles, we choose the next distance $n_{j+1}$, 
at which we pause the particles again,
in terms of $n_j$ and $k_j$.
We iterate this procedure until there are less than $n^{1/(d+1)}$ paused particles.
At this point, there are too few particles to matter.

\begin{proof}[Proof of Theorem~\ref{thm:main thm}]
Fix $n>0$. Define $A_j, n_j, P_j,k_j$ as follows:
\begin{itemize}
\item Let $n_0=n$ and $A_0=A_{b_o(n)}(o\until B_o(n))$. 
Let $P_0 = P_{b_o(n)}(o \until B_o(n))$ and
let $k_0 = |P_0|$.

\item For $j \geq 0$, define
$$n_{j+1}=
\left\{
\begin{array}{ll}
n_j+k_j^{1/d} & \text{if }k_j > n^{1/(d+1)} , \\
\infty & \text{otherwise}.
\end{array} \right .$$
Let
$A_{j+1} = A( A_j;P_j \until B_o(n_{j+1}))$.
Let $P_{j+1} = P(A_j;P_j\until B_o(n_{j+1}))$
and let $k_{j+1} =|P_{j+1}|$. 
\end{itemize}
Let $J$ be the (random) first time at which $k_J \leq n^{1/(d+1)}$.
By construction, $A_{j}=A_{J+1}$ for any $j \ge J+1$.
The Abelian property \eqref{eqn:abelian property}
guarantees that $A_{J+1}$ and $A_{b_o(n)}(o)$ have the same law.

By construction, $A_J \subset B_o(n_J)$.
Since $k_J \leq n^{1/(d+1)}$ and $\rho$ is continuous $(C)$,
the $k_J$ last particles cannot grow long arms. Formally,
$A_{J+1} \subset B_o (n_J + c n^{1/(d+1)} ).$


Since $J \le n$, by Lemma \ref{lem:exit prob},
for some $\delta = \delta(\alpha,c,d) < 1$,
\begin{align*}
\Pr [ \exists \ 1 \leq j \leq J \ : \ k_j > (1-\delta)^j k_0 ] & \leq
\Pr [ \exists \ 1 \leq j \leq J \ : \ k_{j} > (1-\delta) k_{j-1} ] \leq n p^{n^{1/(d+1)} } .
\end{align*}
This implies that with probability at least $1-n p^{n^{1/(d+1)}}$,
$$n_J = n+k_0^{1/d}+\cdots+k_{J-1}^{1/d} \le n+k_0^{1/d} \cdot \frac{1}{1-(1-\delta)^{1/d}} ,$$
and if $n_J + C n^{1/(d+1)} > (1+\eps) n$, then
$k_0^{1/d} > (\eps n - C n^{1/(d+1)} ) (1- (1-\delta)^{1/d})$.
So, using $(VG)$, for any $\eps>0$,
\begin{align*}
\set{ A_{b_o(n)}(o) \not\subset B_o((1+ \eps) n) }  \subseteq
\set{ k_0 > c_1  \eps^d b_o(n) } \cup 
\set{ \exists \ 1 \leq j \leq J \ : \ k_j > (1-\delta)^j k_0  } .
\end{align*}
Using the Borel-Cantelli Lemma, we deduce the result easily.
\end{proof}

\paragraph{Acknowledgements}We wish to thank Itai Benjamini for suggesting the problem
to us. This paper was written during the visit of the first
two authors to the Weizmann Institute in Israel. The
first author was supported by the EU Marie-Curie RTN CODY, the ERC AG
CONFRA, as well as by the Swiss {FNS} and the Weizmann institute. 
The fourth author is a Horev fellow and is supported by the Taub Foundation,
and by grants from ISF and BSF.

\begin{flushright}
\footnotesize\obeylines
  \textsc{Universit\'e de Gen\`eve}
  \textsc{Gen\`eve, Switzerland}
  \textsc{E-mail:} \texttt{hugo.duminil@unige.ch}
$ $\\
\textsc{MODAL'X-Universit\'e Paris X}
  \textsc{Paris, France}
  \textsc{E-mail:} \texttt{cyrille.lucas@u-paris10.fr}
$ $\\
   \textsc{Ben Gurion University}
  \textsc{Beer Sheva, Israel}
  \textsc{E-mail:} \texttt{yadina@bgu.ac.il}
 $ $\\ 
  \textsc{Technion-IIT}
  \textsc{Haifa, Israel}
  \textsc{E-mail:} \texttt{amir.yehudayoff@gmail.com}
\end{flushright}


\begin{thebibliography}{10}

\bibitem{antal1996chemical}
Peter Antal and Agoston Pisztora.
\newblock On the chemical distance for supercritical {B}ernoulli percolation.
\newblock {\em Ann. Probab.}, 24(2):1036--1048, 1996.

\bibitem{asselah2010logarithmic}
A.~Asselah and A.~Gaudilliere.
\newblock From logarithmic to subdiffusive polynomial fluctuations for internal
  {DLA} and related growth models.
\newblock {\em Arxiv preprint arXiv:1009.2838}, 2010.

\bibitem{asselah2010sub}
A.~Asselah and A.~Gaudilliere.
\newblock Sub-logarithmic fluctuations for internal {DLA}.
\newblock {\em Arxiv preprint arXiv:1011.4592}, 2010.

\bibitem{barlow}
M.T. Barlow.
\newblock Random walks on supercritical percolation clusters.
\newblock {\em The Annals of Probability}, 32(4):3024--3084, 2004.

\bibitem{blachere2004internal}
S.~Blach{\`e}re.
\newblock Internal diffusion limited aggregation on discrete groups of
  polynomial growth.
\newblock In {\em Random walks and geometry: proceedings of a workshop at the
  Erwin Schr{\"o}dinger Institute, Vienna, June 18-July 13, 2001}, page 377. De
  Gruyter, 2004.

\bibitem{blachere2007internal}
S.~Blach{\`e}re and S.~Brofferio.
\newblock Internal diffusion limited aggregation on discrete groups having
  exponential growth.
\newblock {\em Probability Theory and Related Fields}, 137(3):323--343, 2007.

\bibitem{DF}
P.~Diaconis and W.~Fulton.
\newblock A growth model, a game, an algebra, lagrange inversion, and
  characteristic classes.
\newblock {\em Rend. Sem. Mat. Univ. Pol. Torino}, 49(1):95--119, 1991.

\bibitem{huss13internal}
W.~Huss.
\newblock Internal diffusion-limited aggregation on non-amenable graphs.
\newblock {\em Electronic Communications in Probability}, 13:272--279, 2008.

\bibitem{jerison2010internal}
D.~Jerison, L.~Levine, and S.~Sheffield.
\newblock Internal {DLA} in higher dimensions.
\newblock {\em Arxiv preprint arXiv:1012.3453}, 2010.

\bibitem{jerison2010logarithmic}
D.~Jerison, L.~Levine, and S.~Sheffield.
\newblock Logarithmic fluctuations for internal {DLA}.
\newblock {\em Arxiv preprint arXiv:1010.2483}, 2010.

\bibitem{jerison2011internal}
D.~Jerison, L.~Levine, and S.~Sheffield.
\newblock Internal {DLA} and the {G}aussian free field.
\newblock {\em Arxiv preprint arXiv:1101.0596}, 2011.

\bibitem{lawler1995subdiffusive}
Gregory~F. Lawler.
\newblock Subdiffusive fluctuations for internal diffusion limited aggregation.
\newblock {\em Ann. Probab.}, 23(1):71--86, 1995.

\bibitem{lawler1992internal}
Gregory~F. Lawler, Maury Bramson, and David Griffeath.
\newblock Internal diffusion limited aggregation.
\newblock {\em Ann. Probab.}, 20(4):2117--2140, 1992.

\bibitem{shellef15idla}
E.~Shellef.
\newblock Idla on the supercritical percolation cluster.
\newblock {\em Electronic Journal of Probability}, 15:723--740, 2010.

\end{thebibliography}
\end{document}